\documentclass[11pt]{article}
\usepackage{amsmath,amssymb,amsfonts,amsthm}
\usepackage{float}
\numberwithin{equation}{section}
\newtheorem{theorem}{Theorem}[section]

\newtheorem{lemma}[theorem]{Lemma}
\newtheorem{proposition}[theorem]{Proposition}

\begin{document}
\begin{center}
{\Large{\textbf{A note on the number of edges of the Jaco Graph, $J_n(1), n \in \Bbb N$}}} 
\end{center}
\vspace{0.5cm}
\large{\centerline{(Johan Kok, Vivian Mukungunugwa)\footnote {\textbf {Affiliation of authors:}\\
\noindent Johan Kok (Tshwane Metropolitan Police Department), City of Tshwane, Republic of South Africa\\
e-mail: kokkiek2@tshwane.gov.za\\ \\
\noindent Vivian Mukungunugwa (Department of Mathematics and Applied Mathematics, University of Zimbabwe), City of Harare, Republic of Zimbabwe\\
e-mail: vivianm@maths.uz.ac.zw}}
\vspace{0.5cm}
\begin{abstract}
\noindent Kok et.al. $[3]$ introduced Jaco Graphs (\emph{order 1}).  It is hoped that as a special case, a closed formula can be found for the number of edges of a finite Jaco Graph $J_n(1).$ However, the algorithms discussed in Ahlbach et al.[1] suggest this might not be possible. Finding a closed formula for the number of edges of a Jaco Graph, $J_n(1), n \in \Bbb N$ remains an interesting open problem. In this note we present three alternative, \emph{formula.}\\ \\
\end{abstract}
\noindent {\footnotesize \textbf{AMS Classification Numbers:} 05C07, 05C20, 05C38, 05C75, 05C85} 
\section{Introduction} 
The infinite Jaco graph (\emph{order 1}) was introduced in $[3],$ and defined by $V(J_\infty(1)) = \{v_i| i \in \Bbb N\}$, $E(J_\infty(1)) \subseteq \{(v_i, v_j)| i, j \in \Bbb N, i< j\}$ and $(v_i,v_ j) \in E(J_\infty(1))$ if and only if $2i - d^-(v_i) \geq j.$ The graph has four fundamental properties which are; $V(J_\infty(1)) = \{v_i|i \in \Bbb N\}$ and, if $v_j$ is the head of an edge (arc) then the tail is always a vertex $v_i, i<j$ and, if $v_k,$ for smallest $k \in \Bbb N$ is a tail vertex then all vertices $v_ \ell, k< \ell<j$ are tails of arcs to $v_j$ and finally, the degree of vertex $k$ is $d(v_k) = k.$ The family of finite directed graphs are those limited to $n \in \Bbb N$ vertices by lobbing off all vertices (and edges arcing to vertices) $v_t, t > n.$ Hence, trivially we have $d(v_i) \leq i$ for $i \in \Bbb N.$
\section{Number of edges of a Jaco Graph, $J_n(1), n \in \Bbb N$}
It is hoped that a closed formula can be found for the number of edges of a finite Jaco Graph $J_n(1).$ However, the algorithms discussed in Ahlbach et al.[1] suggest this might not be possible. Finding a closed formula for the number of edges of a Jaco Graph, $J_n(1), n \in \Bbb N$ remains an interesting open problem. We present three alternative, \emph{formula.}
\begin{lemma}
$\epsilon (J_n(1)) = \epsilon (J_{n-1}(1)) + d^-(v_n).$
\end{lemma}
\begin{proof}
Trivial.
\end{proof}
\noindent We now present the adapted Fisher Algorithm. The original Fisher Algorithm is found in [3].
\subsection{ The adapted Fisher Algorithm for $\{J_i(1), i\in\{4,5,6,\cdots, s\in \Bbb N\}$}
\noindent The family of finite Jaco Graphs are those limited to $n \in \Bbb N$ vertices by lobbing off all vertices (and edges arcing to vertices) $v_t, t > n.$ Hence, trivially we have $d(v_i) \leq i$ for $i \in \Bbb N.$ \\ \\
Note that rows 1, 2 and 3 follow easily from the definition. \\ \\
Step 0: Set $j = 4$, then set $i = j$ and $s\geq 4.$ \\
Step 1: Set $ent_{1i} = i.$ \\
Step 2: Set $ent_{2i}=ent_{1(i-1)} - ent_{4(i-1)}.$ (Note that $d^-(v_i)=v(\Bbb{H}_{i-1}(1))=(i-1)- \Delta(J_{i-1}(1))).$\\
Step 3: Set $ent_{3i}=ent_{1i} - ent_{2i}.$ (Note that $d^+(v_i)=i-d^-(v_i)).$\\
Step 4: Set $ent_{5i} = ent_{5(i-1)} + ent_{2i}.$ (Lemma 2.1)\\  
Step 5: Set $j = i+1$, then set $i = j.$  If $i \leq s$, go to Step 1, else go to Step 6.\\
Step 6: Exit. \\ \\
\noindent \textbf{First recursive formula:} Note that Lemma 2.1 provides the \emph{first} recursive formula to determine the number of edges of $J_n(1).$ Using the adapted Fisher Algorithm together with Lemma 2.1 the table below follows easily. \\
\begin{tabular}{|c|c|c|c|c|}
\hline
$\phi (v_i)\rightarrow i\in{\Bbb{N}}$&$d^-(v_i)=\nu(\Bbb{H}_{i-1})$&$d^+(v_i)= i - d^-(v_i)$&$\Delta(J_i(1))$&$\epsilon(J_i(1)) = \epsilon(J_{i-1}(1)) + d^-(v_i)$\\
\hline
1=$f_2$&0&1&0&0\\
\hline
2=$f_3$&1&1&1&1\\
\hline
3=$f_4$&1&2&2&2\\
\hline
4&1&3&2&3\\
\hline
5=$f_5$&2&3&3&5\\
\hline
6&2&4&3&7\\
\hline
7&3&4&4&10\\
\hline
8=$f_6$&3&5&5&13\\
\hline
9&3&6&5&16\\
\hline
10&4&6&6&20\\
\hline
11&4&7&7&24\\
\hline
12&4&8&7&28\\
\hline
13=$f_7$&5&8&8&33\\
\hline
14&5&9&8&38\\
\hline
15&6&9&9&44\\
\hline
16&6&10&10&50\\
\hline
17&6&11&10&56\\
\hline
18&7&11&11&63\\
\hline
19&7&12&11&70\\
\hline
20&8&12&12&78\\
\hline
21=$f_8$&8&13&13&86\\
\hline
22&8&14&13&94\\
\hline
23&9&14&14&103\\
\hline
24&9&15&15&112\\
\hline
25&9&16&15&121\\
\hline
26&10&16&16&131\\
\hline
27&10&17&16&141\\
\hline
28&11&17&17&152\\
\hline
29&11&18&18&163\\
\hline
30&11&19&18&174\\
\hline
31&12&19&19&186\\
\hline
32&12&20&20&198\\
\hline
33&12&21&20&210\\
\hline
34=$f_9$&13&21&21&223\\
\hline
35&13&22&21&236\\
\hline 
\end{tabular}\\ \\ \\ \\
\noindent \textbf{Second formula:} It is well known that $\epsilon (J_n(1)) = \sum\limits^n_{i=1}d^-(v_i).$ Since $d^-(v_n) = n - d^+(v_n)$, the number of edges is also given by $\epsilon(J_n(1)) = \frac{1}{2}n(n-1) - \sum\limits^n_{i=1}d^+(v_i).$ Furthermore, for $n\geq 2$ we have $d^+(v_1) = 1$, so we rather consider $\epsilon (J_n(1)) =(\frac{1}{2}(n(n-1) - 1) - \sum\limits^n_{i=2}d^+(v_i).$\\ \\
Bettina's Theorem [4] provides for a method to determine $d^+(v_i), \forall i \in \Bbb N$.\\ \\
\noindent \textbf{Example 1.} For the Jaco Graph $J_{15}(1)$ we have;\\ $\epsilon (J_{15}(1)) = \frac{1}{2}.15.(15+1) -1 - \sum\limits^{15}_{i=2} = 119 - \sum\limits^{15}_{i=2}d^+(v_i).$\\ \\
\noindent Now, $1=f_2, 2 = f_3, 3 = f_4, 4=f_4 + f_2, 5 = f_5, 6= f_5 + f_2, 7 = f_5 + f_3, 8= f_6, 9= f_6 + f_2, 10 = f_6 + f_3, 11 = f_6 + f_4, 12 = f_6 +f_4 +f_2,13= f_7, 14 = f_7 +f_2,$ and $15 = f_7 +f_3. $\\ \\
\noindent From Bettina's Theorem it follows that;\\ $\sum\limits^{15}_{i=2}d^+(v_i) = f_2 + f_3 + (f_3+f_1) + f_4 + (f_4+f_1) + (f_4 + f_2) + f_5 + (f_5 + f_1) + (f_5 + f_2) +(f_5 + f_3) + (f_5+ f_3+ f_1) +f_6 + (f_6 + f_1) + (f_6 + f_2) = 5f_1 + 4f_2 + 4f_3 + 3f_4 +5f_5 +3f_6 = f_1 +5f_3 +5f_5 +3f_7 = 75.$\\ \\
So, $\epsilon (J_{15}(1)) = 119 - 75 = 44.$\\ \\
\noindent \textbf{Third formula:} The third formula follows from Proposition 2.2 and 2.3.
\begin{proposition}
If  for the finite Jaco Graph $J_n(1), n \in \Bbb N,$ $n$ can be expressed as $n= m_{=d^+(v_n)} + d^+(v_m)$ we have that:\\
$\epsilon(J_n(1)) = \frac{1}{2}(\sum\limits^m_{i=1}i + \sum\limits^{j_{max}}_{i=0} (d^+(v_{m - i})- i)_{d^+ (v_{m - j_{max}}) - j_{max} \geq 1} + d^+(v_m)(d^+(v_m) - 1)).$
\end{proposition}
\begin{proof}
Consider the the Jaco Graph $J_n(1)$ with $ n = m +d^+(v_m).$ Now consider the Jaco Graph $J_m(1).$ From the definition, $J_m(1)$ was obtained by lobbing off the vertices $v_{m+1}, v_{m+2}, ..., v_{m+d^+(v_m)}$ together with all arcs (edges) incident to the said vertices. Reconstructing $J_n(1)$ can follow in three steps.\\ \\
Step 1: Link the arcs $(v_m, v_{m+1}), (v_m, v_{m+2}), ..., (v_m, v_{m + d^+(v_m)})$ to construct the graph $J^*(1).$ After these additional $d^+(v_m)$ arcs are added we have $d(v_m) = m$ in $J^*(1)$  as required by definition. Note that $d^+(v_m)$ will be denoted $(d^+(v_m) - 0).$\\ \\
Step 2: Because vertex $v_m$ is the Jaconian vertex of $J_n(1)$, we link all arcs (edges) to vertices $v_{m+1}, v_{m+2}, ..., v_{m+ d^+(v_m)}$ to ensure the Hope graph of $J_n(1)$ is constructed (See [3].) Hence $d^+(v_m)(d^+(v_m) -1)$ arcs are added.\\ \\
Step 3: If $d(v_{m-1}) < (m-1)$ then exactly $(d^+(v_{m-1}) - 1)$ \emph{bridging arcs} are needed. So link $(v_{m-1}, v_{m+1}), (v_{m-1}, v_{m+2}), (v_{m-1}, v_{m+3}), ..., (v_{m-1}, v_{m+ d^+(v_{(m-1)}-1)}) .$ Recursively we add \emph{bridging arcs} for vertices $v_{m-2}, v_{m-3}, ..., v_{m- j_{max}}$ such that $(d^+(v_{m - j_{max}}) -j_{max}) \geq 1.$\\ \\
Now the original Jaco graph $J_n(1)$ has been reconstructed. From the well-known general result, $\epsilon (G) = \frac{1}{2}\sum\limits_{v\in V(G)}d_G(v),$ the result:\\ \\
$\epsilon(J_n(1)) = \frac{1}{2}(\sum\limits^m_{i=1}i + \sum\limits^{j_{max}}_{i=0} (d^+(v_{m - i})- i)_{d^+ (v_{m - j_{max}}) - j_{max} \geq 1} + d^+(v_m)(d^+(v_m) - 1)),$ follows. 
\end{proof}
\noindent\textbf{Example 2.} Determine $\epsilon(J_{31}(1)).$ Now, $31 = 19 + d^+(v_{19}) = 19 + 12.$ So we have:\\ \\
$\epsilon (J_{31}(1)) = \frac{1}{2}(\sum\limits^{19}_{i=1}i + \sum\limits^{j_{max}}_{i=0} (d^+(v_{19-i}) - i) + d^+(v_{19})(d^+(v_{19}) - 1)) = \frac{1}{2}(\frac{1}{2}(20.19) + [(12-0) +(11-1)+(11-2)+(10-3)+(9-4)+(9-5)+(8-6)+ (8-7)]_{=50} + [12.11]_{=132}) = \frac{1}{2}(190 + 50 + 132) = 186. $\\ \\
\noindent Note that not all $n \in \Bbb N$ can uniquely be written as $n = m + d^+(v_m)$ for $m \in \Bbb N$. From Corollary 3.5 in  $[4]$ the next proposition follows.
\begin{proposition}
If, for the Jaco Graph $J_{n-1}(1)$ the integer $(n-1)$ cannot be expressed as $n-1 = m_{= d^+(v_{n-1})} + d^+(v_m),$ then $n= m_{= d^+ (v_{n-1})} + d^+(v_m) = m_{=d^+(v_n)} + d^+(v_m),$ and $\epsilon(J_{n-1}(1)) = \epsilon(J_n(1)) - d^+(v_m).$
\end{proposition}
\begin{proof}
The result follows from the reverse of Lemma 2.1 because $d^-(v_n) = d^+(v_m).$
\end{proof}
\noindent \textbf{Example 3.} Determine $\epsilon(J_{17}(1)).$ Note that $d^+(v_{17}) + d^+(v_{d^+(v_{17})}) = 18 \neq 17.$ However, $d^+(v_{18}) + d^+(v_{d^+(v_{18})}) = 18,$ so $\epsilon(J_{17}(1)) = \epsilon(J_{18}(1)) - d^+(v_{11}).$ Hence $\epsilon(J_{17}(1)) = 63 -7 = 56.$\\ \\
\noindent \textbf{\emph{Open access:\footnote {To be submitted to the \emph{Pioneer Journal of Mathematics and Mathematical Sciences.}}}} This paper is distributed under the terms of the Creative Commons Attribution License which permits any use, distribution and reproduction in any medium, provided the original author(s) and the source are credited. \\ \\
References (Limited){\footnote {To the honour of my wife, Lizzy Kok. I wish her peace and happiness on her birthday on Sunday, 7 September 2014. My lady thank you for the support you give me, and thank you much for the patience you bear with me whilst I engage in my mathematical endeavours.}} \\ \\
$[1]$ Ahlbach, C., Usatine, J., Pippenger, N., \emph{Efficient Algorithms for Zerckendorf Arithmetic,} Fibonacci Quarterly, Vol 51, No. 13, (2013): pp 249-255. \\  
$[2]$ Bondy, J.A., Murty, U.S.R., \emph {Graph Theory with Applications,} Macmillan Press, London, (1976). \\
$[3]$ Kok, J., Fisher, P., Wilkens, B., Mabula, M., Mukungunugwa, V., \emph{Characteristics of Finite Jaco Graphs, $J_n(1), n \in \Bbb N$}, arXiv: 1404.0484v1 [math.CO], 2 April 2014.\\
$[4]$  Kok, J., Fisher, P., Wilkens, B., Mabula, M., Mukungunugwa, V., \emph{Characteristics of Jaco Graphs, $J_\infty(a), a \in \Bbb N$}, arXiv: 1404.1714v1 [math.CO], 7 April 2014. \\
\end{document}